\documentclass{amsart}
\usepackage[noadjust]{cite}
\usepackage{geometry,graphicx}
\usepackage{amsmath,amssymb,subfig} 
\usepackage{color}
\usepackage[table]{xcolor}
\usepackage[colorlinks,linkcolor=blue,citecolor=red,urlcolor=black,bookmarks=false,hypertexnames=true]{hyperref}
\usepackage{tikz}
\usetikzlibrary{positioning}
\usepackage{comment}
\usepackage{dsfont} 

\usepackage{mathtools}   
      
\theoremstyle{plain}
\newtheorem{theorem}{Theorem}[section]
\newtheorem{prop}[theorem]{Proposition}
\newtheorem{lem}[theorem]{Lemma}
\newtheorem{coro}[theorem]{Corollary}

\theoremstyle{definition}
\newtheorem{definition}[theorem]{Definition}
\newtheorem{remark}[theorem]{Remark}
\newtheorem{question}[theorem]{Question}
\newtheorem{example}[theorem]{Example}

\newcommand{\Z}{{\mathbb Z}}

\newcommand{\R}{{\mathbb R}}
\newcommand{\N}{{\mathbb N}}
\newcommand{\C}{{\mathbb C}}
\newcommand{\mc}{\mathcal}
\newcommand{\A}{{\mc A}}
\newcommand{\ii}{{\mathrm{i}}}
\newcommand{\dd}{{\mathrm{d}}}
\newcommand{\ee}{{\mathrm{e}}}

\newcommand{\sub}{\varrho}
\newcommand{\csub}{\sub^{ }_{\textnormal{\textsf{cyc}}}}

\newcommand{\bbo}{\mathds{1}} 
\newcommand{\Id}{\mathrm{Id}}

\newcommand{\charn}{\chi^{ }_n}

\newcommand{\exend}{\hfill $\Diamond$}

\begin{document}

\title{Spectral Properties of Substitutions on Compact Alphabets}

\author{Neil Ma\~nibo}
\address{Fakult\"at f\"ur Mathematik, Universit\"at Bielefeld, \newline
\hspace*{\parindent}Postfach 100131, 33501 Bielefeld, Germany}
\email{cmanibo@math.uni-bielefeld.de }

\author{Dan Rust}
\address{School of Mathematics and Statistics, The Open University, \newline
 \hspace*{\parindent}Walton Hall, Milton Keynes, MK7
6AA, UK
}
\email{dan.rust@open.ac.uk}

\author{James J. Walton}
\address{Mathematical Sciences Building, University of Nottingham, \newline
\hspace*{\parindent}University Park, Nottingham, NG7 2RD, UK}
\email{James.Walton@nottingham.ac.uk}


\begin{abstract}
We consider substitutions on compact alphabets and provide sufficient conditions for the diffraction to be pure point, absolutely continuous and singular continuous. This allows one to construct examples for which the Koopman operator on the associated function space has specific spectral components. For abelian bijective substitutions, we provide a dichotomy result regarding the spectral type of the diffraction. 
We also provide the first example of a substitution that has countably infinite Lebesgue spectral components and countably infinite singular continuous components. Lastly, we give a non-constant length substitution on a countably infinite alphabet that gives rise to substitutive Delone sets of infinite type. This extends the spectral theory of substitutions on finite alphabets and  Delone sets of finite type with inflation symmetry.
\end{abstract}

\maketitle

\section{Introduction}
Substitutions over infinite alphabets have already been considered in a variety of contexts \cite{F:infinite, DOP:self-induced, Queffelec, RY:profinite}, but very few accounts on their spectral theory are available. One difficulty is that classical methods such as Perron--Frobenius theory are no longer immediately available in the infinite alphabet setting, and so new techniques must be developed \cite{MRW:compact}. Even then, it is possible that the maximal spectral type of the underlying system is generated by an infinite family of functions, which is usually the case.

In this work, we consider continuous substitutions on a compact alphabet and develop the spectral theory for certain families of representative examples, which are infinite alphabet generalisations of the Thue--Morse, period doubling and Rudin--Shapiro substitutions. We demonstrate how similar arguments to those implemented in the finite alphabet setting allow one to completely determine their spectral type and, at the same time, show that new phenomena can occur.

It is well known that there is a correspondence between the diffraction spectrum and dynamical spectrum. Indeed, this has been rigorously established for systems with pure point spectrum in \cite{LMS:pp,BaakeLenz}, for systems with mixed spectrum in the finite local complexity (FLC) setting 
in \cite{BLvE} and for general dynamical systems in \cite{L:sampling}. We employ a diffraction-based approach and study the diffraction of weighted point sets that arise from elements of the subshift. This entails choosing 
an appropriate set of weight functions, which we get for free when the alphabet has an extra group structure. In these cases, the characters 
of the associated group provide complete access to the admissible spectral types. 

We establish sufficient conditions for when the corresponding diffraction measure is pure point and purely singular continuous for bijective, abelian substitutions. These may be considered as compact alphabet generalisations of the Thue--Morse substitution, see Theorem~\ref{thm: cyclic periodic} for equivalent conditions for singular continuity. These spectral results allow one to distinguish bijective abelian substitutions that admit periodic factors. These results generalise the ones in \cite{BG-squiral, Bartlett, Frank-HD}. In the finite alphabet case, it is conjectured that bijective substitutions always have singular spectrum. This has been confirmed for substitutions on a binary alphabet \cite{BG-squiral} and for abelian substitutions \cite{Bartlett} but remains open for the non-abelian case; see \cite{Queffelec,Bartlett}.

We also consider the compact alphabet analogues of the period doubling substitution and show that they have pure point spectrum (both in the dynamical and diffraction sense), which is done in Theorem~\ref{thm: pure point coincidence}, extending a famous result due to Dekking \cite{Dekking}.
The subshifts they generate can be seen as generalisations of Toeplitz shifts over compact alphabets.

To complete the picture, we provide an analogue of the Rudin--Shapiro substitution on a compact alphabet. The alphabet is $\mathcal{A}=S^{1}\times \left\{\mathsf{0},\mathsf{1}\right\}$, and the corresponding dynamical system is a skew product, which is a group extension of the odometer $(\Z_{2},+1)$ by $S^1$; compare \cite{FM:spin, AL:chain, H:cocycle}. 
Its spectral theory is completely determined by the defining spin matrix $W$. From $W$, one can explicitly define the cocycle that induces the $\Z$-action on $\Z_2\times S^{1}$. 
Using the results in \cite{FM:spin}, we prove in Theorem~\ref{thm:spin-dynam} that this example has countably infinite Lebesgue components and countably infinite singular continuous components in its dynamical spectrum. To the authors' knowledge, this is the first substitutive example with such a spectral decomposition; obviously one cannot have this for finite alphabets where the maximal spectral type is generated by only finitely many correlation measures \cite{Bartlett,Queffelec}.  

All of the previously mentioned families are constant-length, and so one can say that, for all of them, infinite local complexity arises in a purely combinatorial fashion. We end with an example of a non-constant length substitution with uniquely ergodic subshift on the alphabet $\mathbb{N}_0 \cup \{\infty\}$. It admits natural tile lengths that are bounded from above and are bounded away from zero \cite{MRW:compact}. This allows one to construct a tiling of $\mathbb{R}$ with infinitely many prototile lengths. 
Here, infinite local complexity also manifests geometrically apart from the combinatorial complexity induced by the alphabet. From this tiling, one can derive a Delone set of infinite type that still has an inflation symmetry in the sense of Lagarias \cite{L:finite}.

\section{Substitutions on compact alphabets}

\subsection{General theory}
Here, we present general notions and results regarding substitutions on compact alphabets. Although this section is self-contained, we refer the reader to \cite{MRW:compact}
for a more detailed treatment and for the proofs of the dynamical results. 

Let $\A$ be a compact Hausdorff space which we call an \emph{alphabet} and whose elements we call \emph{letters}.
Let $\A^+ = \bigsqcup_{n \geqslant 1} \A^n$ denote the set of all finite (non-empty) \emph{words} over the alphabet $\A$, where $\A^n$ has the product topology and $\A^+$ has the topology of a disjoint union.
Let $\A^\ast = \A^+ \sqcup \{\varepsilon\}$, where $\varepsilon$ is the \emph{empty word}.
\emph{Concatenation} is a binary operation $\A^\ast \times \A^\ast \to \A^\ast$ given by $(\varepsilon, u) \mapsto u$, $(u, \varepsilon)\mapsto u$ and $
(u_1 \cdots u_n, v_1 \cdots v_m) \mapsto u_1 \cdots u_n v_1 \cdots v_n$,
where $u = u_1 \cdots u_n \in \A^n$ and $v = v_1 \cdots v_m \in \A^m$.
We write $uv$ for the concatenation of $u$ and $v$.
Note that concatenation is a continuous operation.
Let $\A^{\Z}$ denote the set of bi-infinite sequences over $\A$ with the product topology.
We use a vertical line $|$ to denote the position between the $-1$st and $0$th element of a bi-infinite sequence, and so we write $w = \cdots w_{-2} w_{-1} | w_0 w_1 \cdots$. The space $\A^{\Z}$ is compact by Tychonoff's theorem. The left shift $\sigma\colon \A^{\Z}\to \A^{\Z}$ is defined by $\sigma(w)_{n}=w_{n+1}$.  A subshift $X\subseteq {\A}^{\Z}$ is a subspace of the full shift that is closed and $\sigma$-invariant.

\begin{definition}
Let $\sub \colon \A \to \A^+$ be a continuous function. We call such a function a \emph{substitution} on $\A$. We say $\sub$ is a substitution of \emph{constant length} $n$ if $\sub(\A)\subset \A^n$.
\end{definition}

To avoid trivialities, we always assume that \(|\sub(a)| \geq 2\) for at least one letter \(a \in \A\). If $\A$ is connected, then by continuity, $\sub$ is necessarily of constant length.

\begin{definition}
Let $\sub\colon \A \to \A^+$ be a substitution. We say a word $u \in \A^n$ is \emph{generated} by $\sub$ if there exist $a \in \A$, $k \geqslant 1$ such that $\sub^k(a)$ contains $u$ as a subword. We write
\[
\mc L^n(\sub) := \overline{\{u \in \A^n \mid u \text{ is generated by } \sub\}}
\]
and define $\mc L(\sub): = \bigsqcup \mc L^n(\sub)$. We call $\mc L(\sub)$ the \emph{language} of $\sub$ and call words in $\mc L(\sub)$ \emph{legal}.
We define a closed, shift-invariant subspace of $\A^\Z$ by
\[
X_\sub := \left\{w \in \A^\Z \mid \forall i \leqslant j, \: w_{[i,j]} \in \mc L(\sub)\right\}
\]
and call $X_\sub$ the \emph{subshift} associated with $\sub$ \cite{MRW:compact}.
\end{definition}

We note that, in the previous definition, we have to take the closure to form ${\mc L}^{n}(\sub)$, which is automatically satisfied for finite alphabets. Next, we present the generalisation of the notion of a primitive substitution in the infinite alphabet setting. The following definition is equivalent to that in the work of Durand, Ormes and Petite \cite{DOP:self-induced}; see also Frank and Sadun \cite{PFS:fusion-ILC} and Queff\'{e}lec \cite{Queffelec}.
Let $\sub \colon \A \to \A^+$ be a substitution.
We say $\sub$ is \emph{primitive} if, for every non-empty open set $U \subset \A$, there exists a $p=p(U) \geqslant 0$ such that for all $a \in \A$, some letter of $\sub^{p}(a)$ is in $U$.
A substitution $\sub$ is called \emph{irreducible} if it cannot be restricted to a strictly smaller closed non-empty subalphabet of $\A$. Primitivity implies irreducibility \cite{MRW:compact}, which is consistent with results in the finite-dimensional setting.

Let $E=C(\A)$ be the Banach space of real-valued continuous functions over $\A$ equipped with the sup-norm. The analogue of the transpose of the substitution matrix in the finite alphabet setting is the operator $M\colon E\to E$ given by $Mf(a)=\sum_{b\triangleleft\sub(a)}f(b)$, where the indexing $b\triangleleft\sub(a)$ is taken to include multiplicities. 
Note that $M$ is a bounded and positive operator on $E$, but in the most interesting cases it is a non-compact operator; see \cite{MRW:compact}.

Nevertheless weaker properties give rise to similar spectral consequences. An operator $T\colon E\to E$ with spectral radius $r(T)=1$ is called 
\begin{enumerate}
\item \emph{quasi-compact} if there is a compact operator $K$ and a natural number $n$ such that $\|T^n-K\|<1$;
\item \emph{strongly power convergent} if \(T^n f\) converges for all \(f \in E\) with respect to the norm on \(E\) (i.e., converges uniformly to a continuous function on \(\A\)). Equivalently, \(T^n \to P\) to some bounded operator \(P\) in the strong operator topology;
\item \emph{mean ergodic} if the Ces\`{a}ro average \(A_n(f) \coloneqq \frac{1}{n}\sum_{j=0}^{n-1} T^jf\) converges for each \(f \in E\). Equivalently, \(A_n  \to P\) to some bounded operator \(P\) in the strong operator topology.
\end{enumerate} 
Primitivity and quasi-compactness imply strong power convergence, and strong power convergence implies mean ergodicity. A function $0\neq \ell\in C(\A)$ is called a \emph{natural length function} if $\ell\geqslant 0$, and $M\ell(a)=\lambda\ell(a)$ for each $a\in\A$ where $\lambda\geqslant 0$. We note that necessarily \(\lambda = r > 1\) when \(\sub\) is irreducible \cite[Thm.~4.26, Prop.~4.27]{MRW:compact}.

\begin{theorem}[{\cite[Thm~3.30]{MRW:compact}}]\label{thm:minimal}
If $\sub$ is a primitive substitution on a compact Hausdorff alphabet $\A$, then $(X_\sub,\sigma)$ is minimal.
\end{theorem}

\begin{theorem}[{\cite[Prop.~5.8]{MRW:compact}}]\label{thm:tile-lengths}
 Let $\sub$ be an irreducible  substitution on a compact Hausdorff alphabet $\A$. If the substitution operator $M$ associated with $\sub$ is mean ergodic, $\sub$ admits a unique (up to scaling) natural length function $\ell\in E$ with $\lambda=r(M)>1$. Moreover, $\ell$ is strictly positive.
\end{theorem}

Theorem~\ref{thm:tile-lengths} is the compact alphabet analogue for the existence of a left Perron--Frobenius (PF) eigenvector of the usual substitution matrix. Note that when $\sub$ is constant-length, it automatically admits the natural length function $\ell = \bbo$, where \(\bbo(a) = 1\) for all \(a \in \A\), regardless of compactness properties of $M$. The existence of a unique positive right PF eigenvector associated with $r$ requires one to look at the dual operator $M'\colon E' \to E'$, where $E'$ is the Banach dual of $E$. Here, $M'$ is defined via $M'\mu(f)=\mu(Mf)$ for $\mu\in E'$. One has the following sufficient criteria for the subshift $X_\sub$ to be uniquely ergodic. 

\begin{theorem}[{\cite[Thm.~5.14 Cor.~5.15]{MRW:compact}}]\label{thm:unique-erg}
Let $\sub$ be a substitution on a compact Hausdorff alphabet $\mc{A}$.
Suppose $\sub$ is irreducible and the substitution operator $M$ is strongly power convergent. Then, $(X_\sub,\sigma)$ is uniquely ergodic. \qed
\end{theorem}

\begin{theorem}[{\cite[Thm.~6.17]{MRW:compact}}]\label{thm:unique-erg-CL}
Let $\A$ be a compact Hausdorff alphabet and $\sub$ a substitution over $\A$ of constant length $L$. Suppose $\sub$ is primitive and that the columns $\left\{\sub^{ }_{j}\right\}_{0\leqslant j\leqslant L-1}$ of $\sub$ generate an equicontinuous semigroup. Then, $(X_\sub,\sigma)$ is uniquely ergodic. \qed
\end{theorem}

\subsection{Compactly bijective substitutions}\label{SEC:bijective}
One of our main results concerns substitutions of constant length for which the columns generate a compact abelian group. In this section, we show that such substitutions must take a particular form. Namely, the alphabet can be given a group structure with respect to which the columns act as group translations. The following definition is a natural extension of the corresponding notions from the finite alphabet case.

\begin{definition}
Given a substitution $\sub \colon \A \to \A^+$ of constant length $L$, we say that $\sub$ is \emph{bijective} if each \emph{column} $\sub^{ }_j(a) \coloneqq (\sub(a))_j$, for $0 \leqslant j \leqslant L-1$, is a bijection, and thus by continuity of $\sub$ a homeomorphism of $\mc{A}$. Let $\hom(\A)$ denote the set of homeomorphisms of $\A$ equipped with the compact-open topology. We let $\mc{C}$ denote the subgroup of $\hom(\A)$ generated by the columns of $\sub$. The closure of $\mc{C}$ in $\hom(\mc{A})$ is denoted $\mc{G} \coloneqq \overline{\mc{C}}$. We call $\sub$ \emph{compactly bijective} if $\mc{G}$ is compact. A compactly bijective substitution is called \emph{abelian} if $\mc{G}$ is abelian.
\end{definition}

\begin{remark}
In the finite alphabet case, a bijective substitution is abelian if and only if the group $\mathcal{C}$ generated by the columns is abelian. \exend
\end{remark}

\begin{remark}
In the general case, $\sub$ is abelian if and only if the columns commute and generate a relatively compact subgroup in $\hom(\A)$. Indeed, the algebraic closure of a commuting subset is abelian, as is the topological closure of an abelian subgroup.
\exend
\end{remark}

\begin{lem} \label{lem: transitive action}
Suppose that $\sub$ is primitive and compactly bijective. Then $\mc{G}$ acts transitively on $\A$.
\end{lem}

\begin{proof}
Let $a, b \in \A$. We wish to find some $g \in \mc{G}$ with $g(a) = b$. Let $U \subset \A$ be an open set containing $b$. By primitivity, there is some $N \in \N$ so that $\sub^N(a)$ contains a letter of $U$. Equivalently, we have that $(\sub^{ }_{i_0} \circ \cdots \circ \sub^{ }_{i_{N-1}})(a) \in U$ where each $\sub^{ }_{i}$ is a column of $\sub$. It follows that we may construct a map $f_U \in \mc{C}$ so that $f_U(a) \in U$.

The open sets $U$ containing $b$ form a directed set, and so we have a net $(f_U)$. By compactness, we have that $f_U \to f$ for some $f \in \mc{G}$. Since $(f_U) \to f$ with respect to the compact-open topology it certainly converges pointwise, so in particular $(f_U(a))_U \to f(a)$. Then clearly $f(a) = b$, since $f_U(a) \in U$, where $U$ is an arbitrary open set containing $b$.
\end{proof}

\begin{prop} \label{prop: cpt bij => coset sub}
Suppose that $\sub$ is primitive and compactly bijective. Then $\mc{A} \cong \mc{G}/H$ for some compact subgroup $H \leqslant \mc{G}$. With respect to this identification, we have
\[
\sub [g] = (\sub^{ }_0, \sub^{ }_1, \ldots, \sub^{ }_{L-1}) [g] = [\sub^{ }_0 \cdot g][\sub^{ }_1 \cdot g] \cdots [\sub^{ }_{L-1} \cdot g],
\]
where $[g]$ denotes the coset of $g \in \mc{G}$ in $\mc{G}/H$.
\end{prop}

\begin{proof}
The evaluation map $C(X,Y) \times X \to Y$ is always continuous for $X$ locally compact and Hausdorff and $C(X,Y)$ equipped with the compact-open topology. Hence the action $\mc{G} \times \A \to \A$ of $\mc{G}$ on $\A$, given by $g \cdot a \coloneqq g(a)$, is continuous. 
Since $\mc{G}$ acts transitively on $\mc{A}$ by Lemma \ref{lem: transitive action}, we have an identification between the $\mc{G}$-actions on $\mc{A}$ and on $\mc{G}/H$ by (left) multiplication, where
\[
H \coloneqq \mathrm{stab}(p) \coloneqq \{g \in \mc{G} \mid g(p) = p\}
\]
and $p \in \A$ is an arbitrary basepoint. Stabiliser subgroups of topological group actions are always closed, so $H$ is compact.

These actions are intertwined via the bijection
\[
\phi \colon \mc{G}/H \to \A, \ \ \phi [g] \coloneqq g(p)
\]
with inverse $a \mapsto [g]$, where $g$ is such that $g(p)=a$ (uniquely defined up to coset representative).
The map $\mc{G} \to \A$ defined by $g \mapsto g(p)$ is continuous, since, as above, evaluation is continuous. Hence, this descends to a continuous map on the quotient $\mc{G}/H$, so $\phi$ is continuous. By assumption, $\A$ is Hausdorff and, since $\mc{G}$ is compact, so is $\mc{G}/H$, hence $\phi$ is a homeomorphism.

By assumption, we have that $\sub = (\sub^{ }_0, \sub^{ }_1, \ldots, \sub^{ }_{L-1})$, where each $\sub^{ }_i$ is a column and thus is an element of $\mc{G}$. Finally, writing $\sub$ as a substitution on $\mc{G}/H$ by the identification $\phi$, we have $\sub[g]_i = \sub(g(p))_i =\sub^{ }_i(g(p)) = [\sub^{ }_i\cdot g]$ and the result follows.
\end{proof}

\begin{coro}\label{coro:bijective-abelian}
Let $\sub$ be a primitive, compactly bijective, abelian substitution. Then, we may equip $\A$ with the structure of a compact group so that each column of the substitution is a group rotation. In particular, there exist $\beta_i \in \A$ so that for all $a \in \A$
\[
\sub(a) = (\beta_0 \cdot a) (\beta_1 \cdot a) \cdots (\beta_{L-1} \cdot a).
\]
\end{coro}

\begin{proof}
In applying Proposition \ref{prop: cpt bij => coset sub} we have $H = \mathrm{stab}(p)$ where $p$ is an arbitrary base point. Then $H = \{\text{id}\}$ since $\sub$ is abelian. Indeed, let $g \in \mathrm{stab}(p)$ and $x \in \A$ be arbitrary. By Lemma \ref{lem: transitive action}, the action of $\mc{G}$ on $\A$ is transitive, so there is some $h \in \mc{G}$ with $h(p) = x$. Then $g(x) = (gh)(p) = (hg)(p) = h(p) = x$ and thus $g = \Id$. It follows from Proposition \ref{prop: cpt bij => coset sub} that we may identify $\A \cong \mc{G}$ and that each column is a group rotation.
\end{proof}

It follows that a constant-length substitution is  compactly bijective and abelian if and only if we may equip $\A$ with the structure of an abelian topological group so that each column of $\sub$ is a group translation. We will assume this structure for
Section~\ref{sec:diff-bij} below. 

\begin{definition}
Let $\sub$ be a substitution on a compact alphabet $\A$ and let $X_\sub$ be its subshift.
We call $x \in X_\sub$ a \emph{pseudo-fixed point} of $\sub$ if $\sub(x)=\sigma^{s}(x)$, where $0\leqslant s\leqslant L-1$.  
\end{definition}

For finite alphabets, the existence of a bi-infinite fixed point (of some power of the substitution) is guaranteed. That is, there is always some $x\in X_\sub$ such that $\sub^{p}(x)=x$ for some power $p\in\mathbb{N}$, which is convenient for diffraction. This fails in general for infinite alphabets.
Here, we show that for our purposes, it suffices to consider pseudo-fixed points, as the associated recurrence relations are similar to those for two-sided fixed points. Pseudo-fixed points can be seen as true fixed points of the corresponding affine substitution; see \cite[Sec.~4.8.1]{BG:book}.

\begin{prop}\label{prop: pseudo-fixed point}
Let $\sub$ be a primitive constant-length substitution on a compact Hausdorff abelian group, such that all the columns are group translations. Then there exists another such substitution $\sub'$ such that  $X_{\sub'} = X_\sub$ and $\sub'$ admits a pseudo-fixed point in $X_\sub$.  
\end{prop}

\begin{proof}
We write $\sub\coloneqq \sub^{ }_0\,\sub^{ }_1\,\cdots\,\sub^{ }_{L-1}$, where $\sub_i\colon \A\to \A$ are group translations. For $m\in\mathbb{N}$ we can define $\sub'\coloneqq \sub^m \beta^{-1}_s$, where $\sub^m_s(\theta)=\theta\beta_j$ for some $0\leqslant s \leqslant L^m-1$, so that $\sub'_s=\text{id}$. We see that $\sub$ and $\sub'$ define the same subshift because they generate the same language. More explicitly, one has $\big(\sub'\big)^n(\theta\beta^n_j)=\sub^{mn}(\theta)$ for $\theta\in\A$ and $n\in\mathbb{N}$.
For large enough $m$, we can take $1 \leqslant j \leqslant L^m -2$ so that, by construction, $\sub'$ has a fixed internal letter at the $j$th position, i.e., $\sub'(\theta)_j=\theta$. Using this, we construct a sequence of points that converges to a pseudo-fixed point. 

Let $x^{(0)}\in X$ be such that $x^{(0)}_{[-j,L-1-j]}=\sub(\theta)$ for some $\theta\in\A$. It then follows that $x^{(0)}_0=\theta$. If we substitute, we get that  $\big(\sigma^{-s}\circ\sub'(x_0)\big)_{[-j,L-1-j]}=\sub'(\theta)$. One can then consider the sequence
$\left\{x^{(n)}\right\}_{n\geqslant 0}$, with $x^{(n)}=(\sigma^{-s}\circ\sub')^{n}(x^{(0)})$. Due to the growing nested subwords around $0$, this sequence converges to a point $x\in X$, which by construction satisfies $\big(\sigma^{-s}\circ \sub'\big)(x)=x$ and so $x$ is a pseudo-fixed point for $\sub'$. 
\end{proof}
Hence for us, without loss of generality, we may assume that $\sub$ admits a pseudo-fixed point.

\begin{example} Let $\A=S^1$ and
consider the substitution $\sub\colon [\theta]\mapsto[\theta\alpha]\,[\theta\beta]\,[\theta\gamma]$, with $\alpha,\beta,\gamma\in \A$. We can instead consider \[\sub\coloneqq \sub'\colon [\theta]\mapsto [\theta\alpha\beta^{-1}]\,[\theta]\,[\theta\gamma\beta^{-1}],\]
which has a fixed internal letter at position $s=1$. 
Fix $\theta\in\A$ and let $x^{(0)}$ be the bi-infinite word given by $
x^{(0)}=\cdots [\theta\alpha\beta^{-1}]\mid[\theta]\,[\theta\beta^{-1}\gamma]\cdots$.
Here, the vertical line $|$ signifies the location of the origin. Substituting and shifting by $-s=-1$ yields
\[
x^{(1)}=\big(\sigma^{-1}\circ\sub\big)(x^{(0)})=\cdots [\theta\alpha\gamma^{-2}\beta][\theta\alpha\beta^{-1}]\mid[\theta][\theta\beta^{-1}\gamma][\theta\alpha\beta^{-2}\gamma]\cdots.
\]
By iterating the process, one fixes the letters in more positions, eventually exhausting all $m \in \Z$, which yields the limit $x$ with $\sub(x)=\sigma(x)$.  \exend
\end{example}

\section{Spectral theory}

\subsection{Point sets, measures and diffraction}

In this section, we following the exposition in \cite{BG:book} on diffraction. 
A set $\varLambda\subset \mathbb{R}^d$ is called a \emph{Delone set} if it is uniformly discrete (i.e., there is an $r>0$ such that every ball of radius $r$ contains at most one point in $\varLambda$) and it is relatively dense (i.e., there is  an $R>0$ such that every ball of radius $R$ contains at least two points in $\varLambda$). In this work, we only deal with Delone sets in dimension $d = 1$. A Delone set is called a \emph{Meyer set} if the Minkowski difference $\varLambda-\varLambda$ is also a Delone set.
A weighted Dirac comb supported on $\varLambda$ is given by $\omega=\sum_{z\in \varLambda} f(z)\delta_z$, with $f\colon \varLambda\to \mathbb{C}$, where $\delta_z$ is a Dirac mass at $z$. This is an unbounded (Radon) measure, i.e., a complex-valued linear functional on the space $C_{c}(\R)$ of compactly supported functions on $\R$.

For our spectral analysis, we need to build a weighted Dirac comb from an element $w\in X_\sub$. For this, we need the intermediate step of building a tiling $\mathcal{T}_w$ of $\mathbb{R}$ from $w$. 
In this work, we  restrict to substitutions of constant length and so we associate to each letter $a\in\A$ a tile $\mathfrak{t}_{a}$ of length $1$. 
Next, we derive a coloured point set $\varLambda_w=\left\{(a_z,z)\mid a_z\in \A,\, z\in \varLambda \subset \mathbb{R}\right\}$ from $\mathcal{T}_w$. To do this, we identify the location of a tile in $\mathcal{T}_w$ to be the location of its left endpoint. Since every tile has length $1$, $\text{supp}(\varLambda_w)=\mathbb{Z}$. 
The coloured point set $\varLambda_w$ then satisfies 
$a_z=w_z$, where $w_z$ is the letter at position $z$ in $w\in X_\sub$. Using $\varLambda_w$, we can define a weighted Dirac comb $\omega$ via $\omega=\sum_{z\in \mathbb{Z}} u(w_z)\delta_{z}$, where $u\colon \A\to \mathbb{C}$ is a continuous (hence bounded) function. 

The autocorrelation measure $\gamma$ associated with $\omega$ is given by $\gamma=\omega \circledast \widetilde{\omega}$, where $\circledast$ is the Eberlein convolution, i.e.,
\[
\gamma=\omega \circledast \widetilde{\omega}= \lim_{R\to \infty} \frac{1}{2R+1} \omega|^{ }_{[-R,R]}\ast \widetilde{\omega|^{ }_{[-R,R]}}=\lim_{R\to\infty} \gamma^{(R)}.  
\]
Here, $\ast$ is the usual convolution for finite measures and $\mu|^{ }_{[-R,R]}$ is the restriction of a measure $\mu$ on $[-R,R]$. The twisted measure $\widetilde{\mu}$ can be defined via test functions: $\widetilde{\mu}(f)=\overline{\mu(\widetilde{f})}$ with $\widetilde{f}(x):=\overline{f(-x)}$ for $C_{c}(\R)$.

It follows from a general result regarding lattice-supported weighted Dirac combs that the sequence $\left\{\gamma^{(R)}\right\}$ admit at least one limit point $\gamma_u=\sum_{m\in \Z}\eta_u(m)\delta_m$ \cite{Baake,BaakeMoody}. For a fixed weight function $u$ and $m \in \Z$, the autocorrelation coefficient $\eta^{ }_u(m)$ is given by
\begin{equation}\label{eq:fourier-coeff}
\eta^{ }_u(m)=\lim_{N\rightarrow\infty}\frac{1}{2N+1}\sum^{N}_{j=-N}u(w^{ }_j)\overline{u(w^{ }_{j+m})}.
\end{equation}
When the subshift is uniquely ergodic, one has that $\gamma^{ }_u$ is the unique autocorrelation for all elements in the subshift $X_\sub$, which is true for all cases treated here. 

\begin{definition}
Given a weighted Dirac comb $\omega$, its
\emph{diffraction measure} is the Fourier transform $\widehat{\gamma^{ }_{u}}$ of its autocorrelation $\gamma^{ }_u$.
\end{definition}

Unless stated otherwise, we suppress the subscript $u$ and assume that a weight function has been chosen. 
The diffraction $\widehat{\gamma}$ is a positive measure on $\R$, which by the Lebesgue decomposition theorem admits the splitting
\[
\widehat{\gamma}=\widehat{\gamma}^{ }_{\textsf{pp}}+\widehat{\gamma}^{ }_{\textsf{ac}}+\widehat{\gamma}^{ }_{\textsf{sc}}
\]
into pure point, absolutely continuous and singular continuous components, not all of which vanish.
One main question in diffraction theory is determining whether a certain component is present. 

\begin{remark}
Since $\omega$ is supported on $\Z$, the diffraction $\widehat{\gamma}$ is $\Z$-periodic, i.e., one has $\widehat{\gamma}=\delta_{\Z}\ast\widehat{\gamma}_{\text{FD}}$, where $\widehat{\gamma}^{ }_{\text{FD}}$ (called the fundamental diffraction) is a positive measure on $\mathbb{T}$; see \cite{BLvE}. To determine the spectral type of $\widehat{\gamma}$ is suffices to look at $\widehat{\gamma}^{ }_{\text{FD}}$, whose Fourier coefficients are given by $\eta_u(m)$ in Eq.~\eqref{eq:fourier-coeff}. \exend
\end{remark}

Since the main classes of examples we deal with in this work are generalisations of the period doubling, Thue--Morse, and Rudin--Shapiro substitutions, we include them here and their corresponding spectral properties for comparison. All of them are primitive and hence uniquely ergodic. For more details on the respective diffraction measures, we refer the reader to \cite{BG:book} and \cite{Queffelec}.

\begin{example}\label{ex: finite alph}
$\text{ }$
\begin{enumerate}
\item \textbf{(Period doubling)} Consider the binary alphabet $\A=\left\{a,b\right\}$. The \emph{period doubling} substitution is given by 
$\sub_{\text{pd}}\colon a\mapsto ab, b\mapsto aa$. 
For any element $w$ in the subshift $X_{\text{pd}}$ and \emph{any} non-zero choice of weight function $u\colon \A\to \mathbb{C}$, the diffraction $\widehat{\gamma^{ }_u}$ is pure point. 
\item \textbf{(Thue--Morse) }
Consider the binary alphabet $\A=\left\{a,b\right\}$. The \emph{Thue--Morse} substitution is given by 
$\sub_{\text{TM}}\colon a\mapsto ab, b\mapsto ba$. 
For any element $w$ in the subshift $X_{\text{TM}}$ and for the weight function $u\colon \A\to \mathbb{C}$ given by 
$u(a)=1,\, u(b)=-1$, the diffraction $\widehat{\gamma^{ }_u}$ is singular continuous with respect to Lebesgue measure. 
\item\textbf{(Rudin--Shapiro)}
Consider the quaternary alphabet $\A=\left\{a,b,\overline{a},\overline{b}\right\}$. The \emph{Rudin--Shapiro} substitution is given by 
$\sub_{\text{RS}}\colon a\mapsto ab, b\mapsto a\bar{b}, \bar{a}\mapsto \bar{a}\bar{b}, \bar{b}\mapsto \bar{a}b$. 
For any element $w$ in the subshift $X_{\text{RS}}$ and for the weight function $u\colon \A\to \mathbb{C}$ given by
$u(a)=u(b)=1,\, u(\bar{a})=u(\bar{b})=-1$, the diffraction $\widehat{\gamma^{ }_u}$ is absolutely continuous with respect to Lebesgue measure.

\end{enumerate}
\end{example}

It follows from a general result of Strungaru \cite[Thm.~4.1]{Strungaru} for weighted Dirac combs with Meyer set support $\varLambda$  that 
the autocorrelation $\gamma$ admits the generalised Eberlein decomposition
\begin{equation}\label{eq: generalised Eberlein}
\gamma=\gamma^{}_{s}+\gamma^{}_{0a}+\gamma^{}_{0s},
\end{equation}
where $\gamma^{}_{s},\gamma^{}_{0a}$, and $\gamma^{}_{0s}$ are pure point measures supported on $\varLambda-\varLambda$, which are in one-to-one correspondence with the measures in the Lebesgue decomposition of $\widehat{\gamma}$. That is,
\[
\widehat{\gamma^{}_{s}}=\widehat{\gamma}^{ }_{\textsf{pp}},\quad\quad\quad \widehat{\gamma^{}_{0a}}=\widehat{\gamma}^{ }_{\textsf{ac}},\quad \quad \quad
\widehat{\gamma^{}_{0s}}=\widehat{\gamma}^{ }_{\textsf{sc}}. 
\]
Here, $\gamma^{ }_{s}$ is the \emph{strongly almost periodic} part.
If it can be shown that $\gamma$ is a strongly almost periodic measure, it follows that $\widehat{\gamma}$ is a pure point measure. In our setting, the set $P_\varepsilon$ of $\varepsilon$-almost periods for $\eta$ is given by 
\begin{equation}\label{eq: set of epsilon almost periods}
P_{\varepsilon}=\big\{m \in \Z \mid |\eta(0)-\eta(m)|^{1/2}<\varepsilon\big\}. 
\end{equation}
An equivalent criterion for $\gamma$ to be strongly almost periodic is given by the following general result.

\begin{theorem}[{\cite[Thm.~5]{BaakeMoody}}]\label{thm: epsilon almost periods}
Let $\omega$ be a translation-bounded measure on  a $\sigma$-compact locally compact abelian group $G$ whose autocorrelation measure $\gamma^{ }_{\omega}$ exists and is of the form $\gamma_{\omega}=\sum_{m\in\Delta} \eta(m)\delta_m$, where $\Delta$ is Meyer. Then the following are equivalent:
\begin{enumerate}

\item The set $P_{\varepsilon}$ of $\varepsilon$-almost periods of $\eta(m)$ is relatively dense in $\Delta$.
\item  $\gamma_{\omega}$ is strongly almost periodic.
\item  $\widehat{\gamma_{\omega}}$ is pure point.
\end{enumerate}
\end{theorem}

Note that the boundedness of the weight function $u$ implies the translation boundedness of the weighted Dirac comb $\omega$. Moreover, $\Z$ is trivially a Meyer set. We use these characterisations mentioned above to show that certain diffraction measures are pure point or purely continuous (i.e., no pure point component); see Theorem~\ref{thm: absence of pp eta} and Theorem~\ref{thm: pure point coincidence}.

\begin{remark}
Recently, Lenz, Spindeler and Strungaru proved an equivalent condition for pure point diffraction at the level of the measure $\omega$, which requires $\omega$ to be \emph{mean almost periodic} \cite{LSS:mean-ap}. Here, we stick with the criterion in Theorem~\ref{thm: epsilon almost periods} because the autocorrelation coefficients are always complex-valued, meaning we can leverage the metric in $\mathbb{C}$, whereas mean almost periodicity requires comparing $w_m$ and $w_{j+m}$, which might need some extra work if the alphabet is not metrisable. \exend
\end{remark}

\subsection{Dynamical spectrum}

Diffraction is intimately connected to the dynamical spectrum of $X_\sub$. By dynamical spectrum, we mean the spectrum of the Koopman operator $U\colon L^{2}(X_\sub,\mu)\to L^{2}(X_\sub,\mu)$
associated with the shift action, with $U(f)=f\circ \sigma$ and where $\mu$ is a $\sigma$-invariant measure on $X_\sub$. Let $\mathcal{M}^{+}(\mathbb{T})$ denote the set of positive measures on $\mathbb{T}$.
For $f\in L^{2}(X_\sub,\mu)$, its \emph{spectral measure} is the unique measure $\rho_f\in \mathcal{M}^{+}(\mathbb{T})$ that satisfies $\left\langle f\mid U^{n}f\right\rangle=\int_{\mathbb{T}} z^n \dd\rho_f(z)$ for $n\in\Z$.
A measure $\rho_{\max}\in \mathcal{M}^{+}(\mathbb{T})$ is called the \emph{spectral measure of maximal type} of $U$ if $\rho_f\ll \rho_{\max}$ for all $f$, where $\ll$ denotes the absolute continuity relation for measures. By a slight abuse of notation, when $H\subseteq L^{2}(X,\mu)$ is a translation-invariant subspace, we write $\rho^{H}_{\max}$ for the maximal spectral type of $H$, where $\rho_f\ll \rho^{H}_{\max}$ for all $f\in H$. 
Equivalently, this is the maximal spectral type of the subrepresentation $U|_{H}$.

\section{Diffraction for substitutions on infinite alphabets}

\subsection{Compactly bijective substitutions}\label{sec:diff-bij}
In this section, we investigate the diffraction measure of abelian bijective substitutions, which are infinite alphabet generalisations of the Thue--Morse substitution in Example~\ref{ex: finite alph}.
From Corollary~\ref{coro:bijective-abelian}, we can assume that the alphabet $\A$ is a compact Hausdorf abelian group and the columns $\sub_j$ are group translations.
We let $S^1 \subset \C$ denote the unit circle, also considered as an abelian group. A \emph{group character} $\chi$ is a continuous group homomorphism $\chi\colon \A\rightarrow S^1$. Here, we will use the characters $\chi$ as our weight functions.
Let $w$ be a pseudo-fixed point of $\sub$ and $\chi\in \widehat{\A}$ be a  character.  
We consider the weighted Dirac comb $\omega_{\chi}=\sum_{m\in \Z}\chi(w^{ }_m)\delta_m$. The following result provides a sufficient condition for absence of absolutely continuous diffraction; compare \cite[Thm.~6]{CKM:q-mult} for one-sided sequences.

\begin{prop}\label{prop: bijective absence of ac}
Let $\sub$ be a substitution of constant length $L$ on a compact Hausdorff abelian group $\A$ such that the columns $\sub_j$ are group translations for $0\leqslant j\leqslant L-1$ and let $w$ be a pseudo-fixed point of  $\sub$. 
Then, for any  $\chi\in\widehat{\A}$, the  diffraction of the weighted Dirac comb $\omega=\sum_{m \in \Z}\chi(w^{ }_m)\delta_m$ has no  absolutely continuous component, i.e., $\widehat{\gamma}_{\textsf{ac}}=0$.  
\end{prop}

\begin{proof}
Let $\big\{\beta^{ }_{j}\big\}_{0\leqslant j\leqslant L-1}$ be the translations defining $\sub$.
Since $w$ is a pseudo-fixed point of $\sub$, one has $\sub(w)=\sigma^{s}(w)$ for some $0\leqslant s\leqslant L-1$, meaning the entries of $w$ satisfy the recurrence
$w^{ }_{Lm+k-s}=w^{ }_{m}\beta^{ }_{k}$, 
for all $0\leqslant k\leqslant L-1$ and $m \in \Z$.  
Fix $0\leqslant k\leqslant L-1$. For any $m \in \Z$, one has 
\begin{align}
\eta(Lm+k)&=\lim_{N\rightarrow\infty}\frac{1}{2N+1}\sum_{r=-s}^{L-1-s}\sum_{j=-\lfloor \frac{N}{L}\rfloor-\varphi^{-}_{r}}^{\lfloor \frac{N}{L}\rfloor+\varphi^{+}_r} \chi(w^{ }_{Lj+r})\overline{\chi(w^{ }_{L(j+m)+r+k})}\nonumber\\
&=\lim_{N\rightarrow\infty}\frac{1}{2N+1}\sum_{r=-s}^{L-1-s}\sum_{j=-\lfloor \frac{N}{L}\rfloor-\varphi^{-}_{r}}^{\lfloor \frac{N}{L}\rfloor+\varphi^{+}_r} \chi\Big(w^{ }_j w^{-1}_{j+m+f(r,k)}\Big)\chi(\beta_{(r+s)\,\text{mod}\,L})\overline{\chi(\beta_{(r+k+s)\,\text{mod}\, L})}\nonumber\\
&=\frac{1}{L}\sum_{r=-s}^{L-1-s}\eta\Big(m+f(r,k)\Big)\chi(\beta^{ }_{(r+s)\,\text{mod}\,L})\overline{\chi(\beta_{(r+k+s)\,\text{mod}\,L})}\label{eq: recursion bijective general},
\end{align}
where 
\begin{equation}\label{eq: f and phi}
f(r,k)\coloneqq \begin{cases}
0, & r+k\leqslant L-1-s,\\
\Big\lfloor\frac{r+k+s}{L}\Big\rfloor, & \text{otherwise},
\end{cases} \quad \text{and}\quad 
\varphi^{\pm}\coloneqq \begin{cases}
1, & N-L\lfloor\frac{N}{L}\rfloor\mp r\geqslant L,\\
0,& 0 \leqslant N-L\lfloor\frac{N}{L}\rfloor\mp r\leqslant L-1,\\
-1,& \text{otherwise}.
\end{cases} 
\end{equation}
Note that these linear recurrences are homogeneous and that,  for $0\leqslant m\leqslant L-1$, the values of $\eta(m)$ depend solely on the values $\eta(0)$ and $\eta(1)$. 
In particular, an explicit computation yields
\begin{align}
\eta(1)&=\frac{\eta(0)\sum_{r=0}^{L-2}\chi(\beta^{ }_{r})\overline{\chi(\beta_{(r+1)\,\text{mod}\,L})}}{L\big(1-\frac{1}{L}\chi(\beta{ }_{L-1})\overline{\chi(\beta^{ }_0)}\big)}\quad \label{eq: recursion L one}\\
\eta(L-1)&=\frac{1}{L}\eta(0)\chi(\beta^{ }_0)\overline{\chi(\beta^{ }_{L-1})}+\frac{1}{L}\sum^{L-1}_{r=1}\eta(1)\chi(\beta^{ }_{r})\overline{\chi(\beta_{(r+L-1)\,\text{mod}\,L})}\label{eq: recursion L minus one}\\
\eta(L^{\ell}m)&=\eta(m),\quad \text{for any } m \in \Z \text{ and }\ell\in\mathbb{N}\label{eq: recursion L zero}. 
\end{align}
The key takeaway here is that the relations \eqref{eq: recursion L one}-\eqref{eq: recursion L zero} no longer depend on $s$. Hence, our analysis proceeds as if we had a true fixed point.
Due to the homogeneity of the relations, and the mutual measure-theoretic orthogonality of $\big(\widehat{\gamma}\big)_{\textsf{pp}},\big(\widehat{\gamma}\big)_{\textsf{ac}}$ and $\big(\widehat{\gamma}\big)_{\textsf{sc}}$, one finds that the individual components $\gamma^{ }_s,\gamma^{ }_{0a}$ and $\gamma^{ }_{0s}$ of the autocorrelation satisfy these relations independently; compare \cite[Prop.~2]{BG-squiral}. 

We then consider $\gamma^{ }_{0a}=\sum_{m \in \Z}\eta_{\textsf{ac}}(m)\delta_m$ which gives rise to the $\textsf{ac}$ component. By Eq.~\eqref{eq: recursion L zero} and the Riemann--Lebesgue Lemma, one has $\eta_{\textsf{ac}}(m)=0$, for all $m\geqslant 1$; see \cite[Prop.~8.6]{BG:book} and \cite[Prop.~3.1]{Frank-HD}. To prove the claim, it remains to show that  $\eta_{\textsf{ac}}(0)=0$. 
If $\eta_{\textsf{ac}}(0)\neq 0$, Eq.~\eqref{eq: recursion L minus one} together with $\eta_{\textsf{ac}}(1)=0$ yields
$\eta_{\textsf{ac}}(L-1)=\frac{1}{L}\eta_{\textsf{ac}}(0)\chi(\beta^{ }_0)\overline{\chi(\beta^{ }_{L-1})}\neq 0$,
which contradicts the previous observation. Hence, $\gamma^{ }_{0a}=0$ and consequently $\widehat{\gamma}^{ }_{\textsf{ac}}=0$. 
\end{proof}

\begin{remark}
The set of recurrences in Eq.~\eqref{eq: recursion bijective general} imply that the sequence $\left\{\eta(m)\right\}_{m\geqslant 0}$ associated with such a substitution is $L$-regular over $\C$ in the sense of Allouche and Shallit; compare \cite[Ex.~14]{AS:regular}. \exend
\end{remark}

We now give a sufficient criterion for the absence of pure point diffraction components.

\begin{theorem}\label{thm: absence of pp eta}
Let $\sub$ be as in Proposition~\textnormal{\ref{prop: bijective absence of ac}} such that $|\eta(1)|<1$. Then, the diffraction $\widehat{\gamma}$ associated with $\omega$ does not have a pure point component. 
\end{theorem}
\begin{proof}
Suppose $\widehat{\gamma}^{ }_{\textsf{pp}}\neq 0$.  Then, by Eq.~\eqref{eq: generalised Eberlein}, $\gamma^{ }_{s}=\sum_{m \in \Z}\eta^{ }_{\textsf{pp}}(m)\delta_m\neq 0$. Because $\gamma^{ }_{s}$ is a strongly almost periodic measure, Theorem~\ref{thm: epsilon almost periods} implies that the set $P_{\varepsilon}$ of $\varepsilon$-almost periods of $\eta^{ }_{\textsf{pp}}(m)$ given in  Eq.~\eqref{eq: set of epsilon almost periods} must be relatively dense, for all $\varepsilon>0$. 
We show that $|\eta(1)|<1$ implies that $|\eta^{ }_{\textsf{pp}}(m)|$ is uniformly bounded away from $|\eta_{\textsf{pp}}(0)|$, which directly contradicts the relative denseness of $P_{\varepsilon}$, for all $\varepsilon\leqslant \varepsilon'$, for some $\varepsilon'$. 

The autocorrelation coefficients of each of the three components of $\gamma$ must satisfy Eq.~\eqref{eq: recursion bijective general} independently since the corresponding component measures are mutually singular. With this, the assumption on $\eta(1)$ and Eq.~\eqref{eq: recursion L one} imply $|\eta^{ }_{\textsf{pp}}(1)|<|\eta^{ }_{\textsf{pp}}(0)|$.  Suppose  $|\eta^{ }_{\textsf{pp}}(0)|-|\eta^{ }_{\textsf{pp}}(1)|>c$, for some $c>0$. 
 Restricting the recurrences in Eq.~\eqref{eq: recursion bijective general} to $\gamma^{ }_{s}$  one gets for  $1\leqslant k\leqslant L-1$, 
\[
|\eta^{ }_{\textsf{pp}}(k)|\leqslant \frac{1}{L}\sum^{L-1}_{r=0}\bigg|\eta^{ }_{\textsf{pp}}\left(f(r,k)\right)\bigg|=\frac{L-k}{L}|\eta^{ }_{\textsf{pp}}(0)|+\frac{k}{L}|\eta^{ }_{\textsf{pp}}(1)|<|\eta^{ }_{\textsf{pp}}(0)|-\frac{k}{L}c
\]
which means 
\begin{equation}\label{eq: ineq zero and k}
|\eta^{ }_{\textsf{pp}}(0)|-|\eta^{ }_{\textsf{pp}}(k)|> \frac{1}{L}c.
\end{equation}
Now let $L\leqslant k=L m+j\leqslant L^2-1$, with $1\leqslant m\leqslant L-2$ and $0\leqslant j\leqslant L-1$. Using the same arguments as in the previous step, one obtains 
\begin{align*}
|\eta^{ }_{\textsf{pp}}(k)|&\leqslant \frac{1}{L}\sum^{L-1}_{r=0}\bigg|\eta^{ }_{\textsf{pp}}\bigg(m+f(r,j)\bigg)\bigg|=\frac{L-j}{L}|\eta^{ }_{\textsf{pp}}(m)|+\frac{j}{L}|\eta^{ }_{\textsf{pp}}(m+1)|\\
&< \frac{L-j}{L}\Big(|\eta^{ }_{\textsf{pp}}(0)|-\frac{m}{L}c\Big)+\frac{j}{L}\Big(|\eta^{ }_{\textsf{pp}}(1)|-\frac{m+1}{L}c\Big)\\
&<|\eta^{ }_{\textsf{pp}}(0)|-\frac{Lm+j}{L^2}c. 
\end{align*}
Since $\frac{1}{L}\leqslant \frac{Lm+j}{L^2} \leqslant \frac{L^2-1}{L^2}$, we recover Eq.~\eqref{eq: ineq zero and k} for all $L\leqslant k\leqslant L^2-1$. It is straight-forward to show via an inductive argument that Eq.~\eqref{eq: ineq zero and k} holds for all $k \in \Z$. This means 
\[
|\eta^{ }_{\textsf{pp}}(0)-\eta^{ }_{\textsf{pp}}(k)|\geqslant \big||\eta^{ }_{\textsf{pp}}(0)|-|\eta^{ }_{\textsf{pp}}(k)|\big| > \frac{1}{L}c, 
\]
which implies $P_{\varepsilon}= \varnothing$, for all $\varepsilon<\sqrt{\frac{c}{L}}$, contradicting the relative denseness property required for pure point diffraction. Hence, $\gamma^{ }_{s}=0$ and $\widehat{\gamma}^{ }_{\textsf{pp}}=0$. 
\end{proof}

The following corollary is immediate from Proposition~\ref{prop: bijective absence of ac} and Theorem~\ref{thm: absence of pp eta}.

\begin{coro}\label{cor: pure sc}
Let $\sub$ be as in Proposition~\textnormal{\ref{prop: bijective absence of ac}}, $w$ a pseudo-fixed point of $\sub$, and $\omega$ the associated weighted Dirac comb via the weight function $\chi\in\widehat{\A}$. If $|\eta(1)|<1$, the diffraction $\widehat{\gamma}$ of $\omega$ is purely singular continuous. \qed
\end{coro}

\begin{lem}\label{lem: eta different column}
Let $\sub$ be as in Proposition~\textnormal{\ref{prop: bijective absence of ac}}. Fix $\chi\in\widehat{\A}$ and suppose $\chi(\beta^{ }_0)\neq\chi(\beta^{ }_{L-1})$. Then, one has $|\eta(1)|< 1$. 
\end{lem}
\begin{proof}
This follows directly from Eq.~\eqref{eq: recursion L one} since, 
$
|\eta(1)|\leqslant \frac{|\eta(0)|(L-1)}{|L-\chi(\beta^{ }_{L-1})\overline{\chi(\beta^{ }_0)}|}<1$,
where the last inequality holds since the  denominator satisfies $|L-\chi(\beta^{ }_{L-1})\overline{\chi(\beta^{ }_0)}|> L-1$, which follows from the assumption $\chi(\beta^{ }_{L-1})\overline{\chi(\beta^{ }_0)}=\chi(\beta^{ }_{L-1}\beta^{-1}_0)\neq 1$. 
\end{proof}

\begin{remark}
Note that when $\A=S^{1}$ the condition in Lemma~\ref{lem: eta different column} is satisfied for non-trivial $\chi\in\widehat{\A}$ whenever $\beta^{ }_0\beta^{-1}_{L-1}$ is an irrational rotation because $\text{ker}(\chi)$ is finite for all non-trivial $\chi\in\widehat{S^{1}}$. \exend
\end{remark}

\begin{theorem}\label{thm: cyclic periodic}
Let $\sub$ be a primitive substitution of length $L$ on a compact  Hausdorff abelian group $\A$ such that all of its columns are group translations. Let $w$ be a pseudo-fixed point of  $\sub$ and consider $\chi\in\widehat{\A}$.
Then, the following are equivalent:

\begin{enumerate}
\item $\widehat{\gamma^{ }_{\chi}}$ is not purely singular continuous;
\item  $|\eta(1)|=1$;
\item The Dirac comb $\omega^{ }_{\chi}=\sum_{m \in \Z} \chi(w_m)\delta_m$ can be generated by a substitution $\csub$ on a finite cyclic group $\A$ of order $n$, where $\sub$ is of the form 
\begin{equation}\label{eq: cyclic periodic}
\csub \colon [\theta]\mapsto [\theta\beta^{ }_0]\, [\theta\beta^{ }_0\beta^{ }_1] \cdots [\theta\beta^{ }_0\beta^{n-1}_1]\, [\theta\beta^{ }_0], 
\end{equation}
where $\beta^{ }_0,\beta^{ }_1\in \A$ and $\textnormal{ord}(\beta_1)=n$;
\item $\omega^{ }_{\chi}$ is periodic;
\item $\widehat{\gamma^{ }_{\chi}}$ is pure point.

\end{enumerate}

\end{theorem}

Before we go to the proof of Theorem~\ref{thm: cyclic periodic}, we will need the following Lemma; see
\cite[Cor.~4.3]{KY:Ellis-bij}.

\begin{lem}\label{lem: cyclic pure point}
The substitution $\csub$ of Eq.~\eqref{eq: cyclic periodic} defines a periodic subshift. 
\end{lem}

\begin{proof}
Every 2-letter subword of a substituted letter is of the form $[\theta][\beta_1 \theta]$ for some $\theta \in \A$. Moreover, given a pair of this form,
\[
\csub([\theta] \bullet [\beta_1 \theta]) = [\theta \beta_0] \cdots [\theta \beta_0] \bullet [\beta_1 \theta \beta_0] \cdots [\beta_1 \theta \beta_0],
\]
where we emphasise the location of the boundary of supertiles. Note that the pair meeting over the supertile boundary is of the same form. Every 2-letter word generated by $\csub$ is either as a subword of some $\csub(\theta)$, or as the pair at the supertile boundary of $\csub(\theta\theta')$, where $\theta \theta'$ is also generated by $\csub$. It follows that all legal 2-letter words are of the form $[\theta][\beta_1 \theta]$. Then periodicity follows from $\textnormal{ord}(\beta_1) = \#\A = n$.
\end{proof}

\begin{proof}[Proof of Theorem~\textnormal{\ref{thm: cyclic periodic}}]
$\text{ }$
\begin{itemize}
\item \textbf{(iii)} $\implies$ \textbf{(iv)} follows from Lemma~\ref{lem: cyclic pure point}. 
\item \textbf{(iv)} $\implies$ \textbf{(v)} is immediate from the periodicity of $\omega^{ }_{\chi}$. 
\item \textbf{(v)} $\implies$ \textbf{(i)} is clear.
\item \textbf{(i)} $\implies$ \textbf{(ii)} follows from Corollary~\ref{cor: pure sc}.
\end{itemize}

It then remains to show that \textbf{(ii)} implies \textbf{(iii)}.
Suppose $|\eta(1)|=1$. Consider the substitution $\sub:=\sub^{ }_0\sub^{ }_1\cdots \sub^{ }_{L-1}$, with $\sub^{ }_i(\theta):=\theta\beta_i$ for $\beta_i \in \A$. From the proof of Proposition~\ref{prop: pseudo-fixed point}, we can assume that $\beta_s=\text{id}$ for some $1\leqslant s\leqslant L-2$, 
since multiplying all $\beta_i$ with $\beta^{-1}_{s}$ yields a substitution which generates the same subshift.  
 From Lemma~\ref{lem: eta different column}, we can assume  that $\chi(\beta^{ }_0)=\chi(\beta^{ }_{L-1})$, since otherwise $|\eta(1)|<1$.
From Eq.~\eqref{eq: recursion L one}, one must then have
\[
\bigg|\sum_{r=0}^{L-2}\chi(\beta_r)\overline{\chi(\beta_{(r+1)\,\text{mod}\,L})}\bigg|=L-1,
\]
which only holds when $\chi(\beta^{ }_{j}\beta^{-1}_{j+1})$ is constant for all $0\leqslant j\leqslant L-2$.  
Since $\beta^{ }_s=\text{id}$, we obtain $\chi(\beta^{-1}_{s+1})=\chi(\beta^{ }_{s+1}\beta^{-1}_{s+2})$. One can show by induction that $\chi(\beta^{ }_{r})=\chi(\beta_{s+1})^{r-s}$ holds for all $s+1\leqslant r\leqslant L-1$. We also have 
$\chi(\beta^{ }_{r})=\chi(\beta_{s+1})^{L-(s+1-r)}$ for $0\leqslant r\leqslant s$. It follows that $\chi(\beta_{s+1})^{L-1}=\text{id}$, so $\chi(\beta_{s+1})\in S^1$ must be a root of unity of order at most $L-1$. 
Now, let $g^{ }_0 =\chi(\beta^{ }_0), g=\chi(\beta^{ }_{s+1})$ and $ n=\text{ord}(\chi(\beta^{ }_{s+1}))$
and consider $\A'=\left\langle g\right\rangle\simeq C_n$. One can check that the substitution $\sub'\colon C_n\to C_n^{+}$ given by 
\[
\sub': [a] \mapsto [a g^{ }_0 ]\, [a g^{ }_0 g]\, [a g^{ }_0 g^2] \cdots [a g^{ }_0 g^{n-1}] \,[a g^{ }_0]
\]
is primitive and yields the exact same recurrence relations for the autocorrelation 
$\eta(m)$ as $\sub$, which completes the proof. 
\end{proof}

As a remarkable consequence, for a pseudo-fixed point $w$ of an abelian bijective substitution $\sub$ on a compact alphabet, if $\chi(w_j)$ admits infinitely many complex values, then $|\eta(1)|<1$ and the diffraction $\widehat{\gamma^{ }_{\chi}}$ is purely singular continuous.

\begin{example}\label{ex: bijective 1d}
Let $\A=S^1$ and 
define the substitution $\sub^{ }_1\colon \A\mapsto \A^{+}$ via
\[
\sub^{ }_1\colon [\theta] \mapsto [\theta] \hspace{2mm}[\theta] \hspace{2mm} [\alpha\theta]\hspace{2mm} [\theta]
\]
where $\alpha\in \A$. 
Consider the  corresponding bi-infinite fixed point $w$ associated with the legal seed $1|1$. It suffices to look at the one-sided fixed point, which satisfies the recurrences
\begin{align*}
w_{4m}&=w_m        &  w_{4m+2}&=\alpha w_{m}\nonumber\\
w_{4m+1}&=w_{m}    &   w_{4m+3}&=w_m  
\end{align*}
where $w_{m}$ is the letter at position $m\in \mathbb{N}$. Here, we write $\alpha=\ee^{2\pi\ii\varphi}$ for some $\varphi\in [0,1)$.

Let $\charn\in \widehat{S^1}$ be given by $\charn(\theta)=\theta^n$ for $\theta\in S^{1}$.
One can show by direct computation that
$
\eta^{ }_n(1)=\eta^{ }_{n}(4^r)=\frac{1}{3}+\frac{2}{3}\cos(2\pi n\varphi)
$
for any $r\in\mathbb{N}$. This implies that $|\eta(1)|=1$ only when $\varphi\in\mathbb{Q}$ with $n\varphi \in \Z$. It follows from Theorem~\ref{thm: cyclic periodic} that, for all $\alpha\in (0,1)$ not satisfying this condition, the diffraction of $\omega^{ }_{\charn}$ is purely singular continuous for any $x\in X_\sub$. When $\alpha$ is an $n$th root of unity, $\omega^{ }_{\charn}=\delta_\Z$, which obviously has pure point diffraction. \exend
\end{example}

\begin{example}[Substitution with a non-trivial pure point factor]
Here, we demonstrate how the characters can reveal interesting factors. Let $\A=C_2\times S^{1}$ with $C_2=\left\{e,g\right\}$ and consider $\sub$, which is given by $\sub\colon [\theta] \mapsto [\theta\beta_0]\,[\theta\beta_1]\, [\theta\beta_2]$ with $\beta_0=\beta_2=(e,1)$ and $\beta_1=(g,\alpha)$, where $\alpha$ is an irrational rotation. Let $\chi_1\otimes\chi_2\in \widehat{\A}\simeq C_2\times \Z$. When $\chi_1$ is non-trivial and $\chi_2$ is trivial, this character induces a factor map to the substitution on $C_2$ given by $\sub'\colon [a ]\mapsto [a]\, [ag] \, [a]$, which by Theorem~\ref{thm: cyclic periodic} is periodic.   
When $\chi_2$ is non-trivial, $\chi_1\otimes\chi_2(w)$ always admits infinitely many complex values due to the irrationality of $\alpha$ and hence must give rise to singular continuous diffraction by Theorem~\ref{thm: cyclic periodic}.  This means that the maximal equicontinuous factor of $\sub$ must be $(\Z_3\times C_2,T_1\times T_2)$, where $\Z_3$ is the $3$-adic integers, $T_1$ is the $+1$-map and $C_2$ is endowed with the group multiplication $T_2: h\to gh$ by $g$; compare \cite{Frank-HD} and Section~\ref{sec:spin}.
\exend
\end{example}

\subsection{Substitutions with coincidences}\label{SEC:coincidence}
\begin{definition}
Let $\sub$ be a constant-length substitution on a compact alphabet $\A$. We say that $\sub$ admits a \emph{coincidence} if for some $0\leqslant j\leqslant L-1$,  $
\sub(\theta)_{j}=\beta_0$, for all $\theta\in\A$, where $\beta_0$ is a constant. 
\end{definition}

Such substitutions are infinite alphabet generalisations of the period doubling substitution in Example~\ref{ex: finite alph}.
As in Proposition~\ref{prop: pseudo-fixed point}, it is routine to show that when one has at least one coincidence and the other columns are group translations, $\sub$ also admits a pseudo-fixed point, so we omit the proof. 

\begin{theorem}\label{thm: pure point coincidence}
Let $\sub$ be a substitution of length $L$ on a compact Hausdorff abelian group $\A$  such that  $\sub(\theta)_{j}$ is either a group translation or is constant, for $0\leqslant j\leqslant L-1$, with at least one constant column. Then, for any $\chi\in\widehat{\A}$, the  weighted Dirac comb $\omega^{ }_{\chi}=\sum_{m\in\Z}\chi(w_m)\delta_m$ arising from a pseudo-fixed point $w$ has pure point diffraction. 
\end{theorem}

\begin{proof}
Without loss of generality, choose $s$ to be the minimal position with a coincidence and let $w\in X_\sub$ be such that $\sub(w)=\sigma^{s}(w)$. Suppose that there are $p$ coincident columns in $\sub$, which are at positions $\big\{c_1=s,c_2,\ldots,c_p\big\}$.

Consider $\eta(Lm)$, which can be written as
\[
\eta(Lm)
=\lim_{N\rightarrow\infty}\frac{1}{2N+1}\sum_{r=-s}^{L-1-s} \sum^{\big\lfloor\frac{N}{L}\big\rfloor+\varphi_r}_{j=-\big\lfloor\frac{N}{L}\big\rfloor-\varphi_r} \chi\left(w^{ }_{Lj+r}w^{-1}_{L(j+m)+r}\right),
\]
where $\varphi_r$ is the same as in Section~\ref{SEC:bijective}.
The constant columns will yield $p$ recurrence relations for $w_j$, which read $w_{Lm+c_i-s}=\beta_{c_i}$. For the remaining positions $0\leqslant k\leqslant L-1$, which are group translations, one has 
$w_{Lm+k-s}=w_m\beta_k$ as in the bijective case.
Taking these recurrences together yields  
\begin{align}
\eta(Lm)&=\lim_{N\rightarrow\infty}\frac{1}{2N+1}\left[p \sum^{\big\lfloor\frac{N}{L}\big\rfloor+\varphi_r}_{j=-\big\lfloor\frac{N}{L}\big\rfloor-\varphi_r} \chi(e)+(L-p) \sum^{\big\lfloor\frac{N}{L}\big\rfloor+\varphi_r}_{j=-\big\lfloor\frac{N}{L}\big\rfloor-\varphi_r} \chi(w^{ }_{j}w^{-1}_{j+m})\right] \nonumber \\ 
&=\frac{p}{L}+\frac{L-p}{L}\eta(m) \label{eq: general LM zero recursion}. 
\end{align}
The unitarity of $\chi$ implies that $\eta(0)=1$. To use Theorem~\ref{thm: epsilon almost periods}, we then need to show that 
\[
P_{\varepsilon}=\big\{m\in \Z\mid |1-\eta(m)|^{1/2}<\varepsilon\big\} 
\]
is relatively dense, for all $\varepsilon>0$. 
To this end, let $\eta(m)=a+\ii b \in \C$. Applying Eq.~\ref{eq: general LM zero recursion} iteratively yields
\[
\eta(L^{\ell}m)=\bigg(\frac{L-p}{L}a_{\ell-1}+\frac{p}{L}\bigg)+\ii\bigg(\frac{L-p}{L}\bigg)^{\ell}b,
\]
where $a_{\ell}=\frac{L-p}{L}a_{\ell-1}+\frac{p}{L}$ with $a_0=a$. Since $|\eta(m)|\leqslant 1$ for all $m$, one has $a,b\in[-1,1]$, implying that, for any fixed $m\in\Z$, 
$\mathfrak{Re}(\eta(L^{\ell}m))\rightarrow 1$ and $\mathfrak{Im}(\eta(L^{\ell}m))\rightarrow 0$
as $\ell\rightarrow\infty$, with uniform convergence in $[-1,1]$. 
Given  $\varepsilon>0$, and $z=L^{\ell}m$ for some $\ell\in\mathbb{N}$ and $m\in \Z$, one has 
\begin{align*}
|1-\eta(z)|^{1/2}&=\Bigg|\bigg(1-\frac{L-p}{L}a_{\ell-1}+\frac{p}{L}\bigg)-\ii\bigg(\frac{L-p}{L}\bigg)^{\ell}b\Bigg|^{1/2} \\
&\leqslant |\varepsilon^{(\ell)}_1+\varepsilon^{(\ell)}_2|^{1/2}\leqslant \sqrt{2}\max\Bigg\{\sqrt{\varepsilon^{(\ell)}_1},\sqrt{\varepsilon^{(\ell)}_2}\Bigg\}.
\end{align*}
One can then choose $\ell_0\in\mathbb{N}$ such that for all $\ell>\ell_0$, $\max\Big\{\sqrt{\varepsilon^{(\ell)}_1},\sqrt{\varepsilon^{(\ell)}_2}\Big\}<\frac{\varepsilon}{\sqrt{2}}$. 
Since this choice of $\ell_0$ can be made to cover all possible values of $a,b\in[-1,1]$, one has $
|1-\eta(z)|^{1/2}<\varepsilon$
for the elements of the set $\mathcal{Z}^{(\ell_0)}\coloneqq \big\{z\in\Z\mid z=L^{\ell}m,\, m\in\Z,\,\ell\geqslant \ell_0\big\}$.
 Clearly, $\mathcal{Z}^{(\ell_0)}$ is relatively dense in $\Z$, and hence $\widehat{\gamma^{ }_{\chi}}$ is pure point by Theorem~\ref{thm: epsilon almost periods}. 
\end{proof}

The next result follows from \cite[Thm.~7]{BaakeLenz}, which establishes the equivalence of a measure-preserving dynamical system $(X,S,\mu)$ having pure point dynamical spectrum and the diffraction being pure point (which in fact holds without any ergodicity assumptions on $\mu$). 

\begin{coro}
Let $\sub$ be as in Theorem~\textnormal{\ref{thm: pure point coincidence}} such that one of the columns is a group translation with dense orbit in $\A$. Then $\sub$ is primitive and $(X_\sub,\sigma)$ is strictly ergodic with pure point dynamical spectrum.  \qed
\end{coro}

Since $\sub$ has a coincidence, say $\sub_j(a)=\beta$ and has a group rotation $\sub_{k}(a)=\alpha\theta$ with dense orbit as another column, given any open set $U\subset \A$, one can find a power $n\in \mathbb{N}$ such that $(\sub_j)^{n}(\beta)$ lies in $U$. This proves that $\sub$ is primitive.
Here, minimality follows from primitivity and unique ergodicity follows again from Theorem~\ref{thm:unique-erg-CL} since the columns of $\sub$ generate an equicontinuous semigroup.

\begin{example}\label{ex: pure point 1d}
Consider the substitution $\sub^{ }_2\colon \A\to \A^{+}$
\[
\sub^{ }_2\colon [\theta] \mapsto [\theta] \hspace{2mm}[1] \hspace{2mm} [\theta\alpha]\hspace{2mm} [\theta],
\]
where $\A = S^1$, and the corresponding bi-infinite fixed point $w$ associated with the legal seed $1|1$. 
For this example, the recurrence relations are given by
\begin{align*}
w_{4m}&=w_m        &  w_{4m+2}&=\alpha w_{m}   \nonumber\\
w_{4m+1}&=1    &   w_{4m+3}&=w_m.  \label{eq: recursion zero}
\end{align*}
As before,  we assign the weight $\charn(\theta)=\theta^n$ to each point of type $\theta$. Theorem~\ref{thm: pure point coincidence} tells us that the corresponding diffraction for $\sub^{ }_2$ is pure point. \exend
\end{example}

\begin{remark}
Substitutions covered in Theorem~\ref{thm: pure point coincidence} generate sequences exhibiting a generalised Toeplitz structure; see \cite[Sec.~4.3]{EG:almost-minimal} and \cite{W:Toeplitz} on generalised Toeplitz sequences over compact alphabets. \exend
\end{remark}

\subsection{A substitution with countably infinite Lebesgue and countably infinite singular continuous spectral components}\label{sec:spin}

Here, we provide a generalisation of the Rudin--Shapiro substitution on infinite alphabets. 
Let $\A=G\times \mathcal{D}$, with $G=S^{1}$ and $\mathcal{D}=\left\{\mathsf{0},\mathsf{1}\right\}$. Let $\alpha\in S^{1}$ be an irrational rotation and consider $\sub\colon \A\to \A^{+}$ given by 
\begin{equation}\label{eq:spin sub}
\sub\colon \begin{cases}
(\theta,\mathsf{0}) \mapsto (\theta,\mathsf{0})\, (-\theta,\mathsf{1}) & \\
(\theta,\mathsf{1}) \mapsto (\theta\alpha,\mathsf{0})\, (\theta\alpha,\mathsf{1}). \\
\end{cases} 
\end{equation}
Such a substitution is called a \emph{spin substitution} by Frank and Ma\~nibo \cite{FM:spin}.
For $\mathsf{a}=(\theta,\mathsf{d})\in \A$, we call $\pi_{G}(\mathsf{a})=\theta$ its \emph{spin} and $\pi_{\mathcal{D}}(\mathsf{a})=\mathsf{d}$ its \emph{digit}. 
The distribution of the spins in the level-1 substituted words are determined by the spin matrix
\[
W=\begin{pmatrix}
1 & -1\\
\alpha & \alpha
\end{pmatrix},
\]
which satisfies $W_{ij}=\pi_{G}(\sub(1,i)_j)$ for $i,j\in\mathcal{D}$.
Primitivity follows from the irrationality of $\alpha$. 
Note that this substitution is recognisable because the image of any letter with digit $\mathsf{1}$ has the same spin for its first and second letter. One can show that the associated subshift is measure-theoretically isomorphic to an $S^1$-extension of the dyadic odometer, i.e., $(X_\sub,\Z,\mu)\simeq(\Z_2\times S^1,\Z,\nu\times\mu_{\text{H}})$, where $\nu$ and $\mu_{\text{H}}$ are the Haar measures on $\Z_2$ and $S_1$, respectively. The skew product which induces the $\Z$-action can be derived directly from $W$; see \cite[Thm.~3.11]{FM:spin}. 
It follows from the classical theory of group extensions that one has the induced splitting $L^{2}(X_\sub,\mu)=\bigoplus_{\chi\in\widehat{S^{1}}} H_{\chi}$
with $H_{\chi}=L^{2}(\Z_2,\nu)\otimes \chi$. The subspaces $H_{\chi}$ are translation-invariant. 
The spectral type of functions in $H_{\chi}$ is determined by $\chi(W):=(\chi(W_{ij}))^{ }_{i,j}\in \text{Mat}(2,S^1)$ via the following result. 

\begin{theorem}[{\cite[Thm.~3.6]{FM:spin}}]\label{thm:spin-general}
Let $\sub$ be a primitive and recognisable spin substitution on $\A=G\times \mathcal{D}$ with spin map $W$, where $G$ is a compact abelian group.
\begin{enumerate}
\item If $\frac{1}{\sqrt{|\mathcal{D}|}}\chi(W)$ is a unitary matrix,  $H_{\chi}$ has Lebesgue spectral type of multiplicity $|\mathcal{D}|$. 
\item If $\chi(W)$ is a rank-1 matrix,  $\chi$ induces a factor map onto a bijective abelian substitution $\sub'$. Moreover, the maximal spectral type of $H_{\chi}$ is either pure point or purely singular continuous, and is absolutely continuous to the maximal spectral type of $\sub'$. \qed 
\end{enumerate} 
\end{theorem}

\begin{prop}\label{prop:spin-diff}
Let $\sub$ be the spin substitution in Eq.~\eqref{eq:spin sub}. Let $\charn\in \widehat{S^{1}}\simeq \Z$, with $\charn(z)=z^n$. Consider the weighted Dirac comb $\omega^{ }_{\charn}=\sum_{m\in\Z} \charn(\pi_{G}(w_m))\delta_m$, where $w\in X_\sub$. Let $\widehat{\gamma}^{(n)}$ be the  diffraction measure of $\omega^{ }_{\charn}$. 
Then one has the following.
\begin{enumerate}
\item If $n=0$, the diffraction $\widehat{\gamma}^{(n)}=\delta_{\mathbb{Z}}$ and hence pure point.
\item If $n\in 2\Z+1$, the diffraction 
$\widehat{\gamma}^{(n)}$ is Lebesgue measure.
\item If $n\in 2\Z\setminus\left\{0\right\}$, the diffraction $\widehat{\gamma}^{(n)}$ is purely singular continuous and is a generalised Riesz product.
\end{enumerate}
\end{prop}

\begin{proof}
The first claim is straightforward. 
Before we proceed, we note that $\widehat{\gamma}^{(n)}=\rho_f\ast \delta_{\Z}$, where $\rho_f$ is the spectral measure of the function $f\colon x\mapsto \charn(\pi_{G}(x_0))$, which is always in $H_{\charn}$. 
To prove the second claim, note that $\frac{1}{\sqrt{2}}\charn(W)$ is  a unitary matrix whenever $n$ is an odd integer. From Theorem~\ref{thm:spin-general}, $H_{\chi}$ contains only functions whose spectral measures are absolutely continuous with respect to Lebesgue measure. 
For the third claim, the matrix $\chi(W)$ is rank-1
whenever $n$ is even. For $n\neq 0$, the character  $\charn$ induces a factor map from $X_\sub$ onto $(S^1)^{\Z}$ which identifies the letters in $\left\{(\zeta \alpha^n,\mathsf{0})\right\}$ with $(1,\mathsf{1})$, where $\zeta$ is an $n$th root of unity. The image of $X_\sub$ under $\charn$ can be realised as the subshift of the bijective substitution on $S^{1}$ given by
$[\theta] \mapsto [\theta]\, [\theta \alpha^n]$,
which is primitive (since $\alpha$ is irrational) and recognisable. By Theorem~\ref{thm: cyclic periodic}, the diffraction measure 
 $\widehat{\gamma}^{(n)}$ is purely singular continuous. The corresponding generalised Riesz product is given by $\rho=\prod_{m\geqslant 0} \frac{1}{2}|1+\alpha^n\ee^{2^{m+1}\pi\ii t}|^2$, seen as a weak-$\ast$ limit of absolutely continuous measures on $\mathbb{T}$, which arises from the relations $w_{2m}=w_m$ and $w_{2m+1}=\alpha^n w_m$; see \cite[Prop.~4.13]{Queffelec}.
\end{proof}

\begin{theorem}\label{thm:spin-dynam}
Let $\sub$ be the spin substitution in  Eq.~\eqref{eq:spin sub}. Consider the function space  $L^{2}(X_\sub,\mu)=\bigoplus_{\charn\in\widehat{S^{1}}} H_{\charn}$. Let $\rho^{(n)}_{\max}$ denote the maximal spectral type of $H_{\charn}$. 
\begin{enumerate}
\item If $n=0$,  $\rho^{(n)}_{\max}$ is pure point.
\item If $n\in 2\Z+1$, $\rho^{(n)}_{\max}$ is Lebesgue measure. 
\item If $n\in 2\Z\setminus\left\{0\right\}$, $\rho^{(n)}_{\max}$  is purely singular continuous.
\end{enumerate}
\end{theorem}
\begin{proof}
It is a well known result for abelian group extensions that the restriction of $U_T$ on any $H_{\charn}$ is spectrally pure, i.e., $\rho^{(n)}_{\max}$  is either pure point, absolutely continuous or singular continuous \cite{H:cocycle}. It then suffices to find the spectral type of a single function $f_n\in H_{\charn}$ to determine that of $\rho^{(n)}_{\max}$. From the proof of Proposition~\ref{prop:spin-diff}, we know that $f_n(x):=\charn(\pi_G(x_0))$ is in $H_{\charn}$. Since the spectral type of $\sigma_{f_n}$ is the same as the spectral type of the diffraction measure $\widehat{\gamma}^{(n)}$, the claim  follows from Proposition~\ref{prop:spin-diff}. 
\end{proof}

\begin{remark}
As a consequence, one can form functions in $L^{2}(X,\mu)$ whose spectral measures have an arbitrary (finite) number of absolutely continuous and singular continuous components. 
As an example, for the hyperlocal function $f(x)=(\chi^{ }_0+\chi^{ }_1+\chi^{ }_3+\chi^{ }_4)(\pi_{G}(x_0))$, the spectral measure $\rho_{f}$ decomposes into $\rho_f=\rho^{ }_0+\rho^{ }_1+\rho^{ }_3+\rho^{ }_4$, where $\rho^{ }_i\perp \rho^{ }_j$ for $i\neq j$ and where $\rho^{ }_0$ is pure point, $\rho^{ }_1$ and $\rho^{ }_3$ are absolutely continuous, and $\rho^{ }_4$ is singular continuous.  \exend
\end{remark}

\subsection{Non-constant length example}

Let $\A=\N_{\infty} = \N_0 \cup\{\infty\}$ denote the one-point compactification of the natural numbers and consider the substitution 
\[
\sub\colon \left\{
\begin{array}{rcl}
0      & \mapsto & 0\ 1 \\
n      & \mapsto & 0\ n\!-\!1\ n\!+\!1 \\
\infty & \mapsto & 0\ \infty\ \infty.
\end{array}\right.
\]
This is a non-constant length substitution. One can easily check that it is primitive, which implies $X_\sub$ is minimal by Theorem~\ref{thm:minimal}. It has been shown in \cite[Ex.~6.9]{MRW:compact} that the corresponding substitution operator is quasi-compact, which together with primitivity, implies that it is strongly power convergent. Note that the substitution is recognisable since every supertile begins with and never ends with $0$. It follows from Theorem~\ref{thm:unique-erg} that the associated subshift $X_\sub$ is uniquely ergodic.  Strong power convergence and primitivity also guarantee the existence of a strictly positive length function by Theorem~\ref{thm:tile-lengths}. In this case, the inflation factor is $\lambda=\frac{5}{2}$ and the length function $\ell\colon \A\to \R$ is given by $\ell(n)=2-2^{-n}$ and $\ell(\infty) = 2$. The letter frequencies are also well defined and are given by $\nu_n=2^{-(n+1)}$ and $\nu_{\infty}=0$. The existence of natural tile lengths allows one to construct a point set $\varLambda$ from the bi-infinite fixed point $w=\sub^{\infty}(\infty\,|\,0)$ given by 
\[
w=\cdots 002010\infty\infty 0\infty\infty010\infty\infty0\infty\infty|010020101013\cdots.
\] 
Now we recover a tiling $\mathcal{T}_w$ from $w$ by replacing each letter in $w$, starting from the origin, by a tile of length $\ell(n)$ if $w_z=n$. As before, we recover a coloured point set $\varLambda_w$ from $\mathcal{T}_w$ by identifying each tile with its left endpoint. The length function is bounded from above and below, with $1= \ell(0) < \ell(n) < \ell(n+1) < \ell(\infty) = 2$ for $n\in\A$.
This means that the point set $\varLambda=\text{supp}(\varLambda_w)$ is both uniformly discrete and relatively dense, and hence is a Delone set. Since there are infinitely many distinct tile lengths, $\varLambda$ has infinite local complexity and hence is not a Meyer set \cite{L:finite}.
The central portion of the associated Delone set is given in Figure~\ref{fig:pointset}.
We call $\varLambda$ a \emph{Delone set of infinite type} with inflation symmetry, i.e., $\lambda\varLambda\subseteq \varLambda$. Lagarias proved in \cite{L:finite} that, if $\varLambda$ is of finite type and has inflation symmetry, then $\lambda$ must be an algebraic integer. For the example above, $\lambda=\frac{5}{2}$ is algebraic but is not an integer. This leads to the following question:

\begin{question}
Let $\sub$ be a primitive substitution on a compact alphabet with inflation factor $\lambda$. Is it possible for $\lambda$ to be a transcendental number?
\end{question}

\begin{figure}[!h]
\begin{center}
\begin{tikzpicture}[
roundnode/.style={circle, draw=black!60, fill=gray!5, very thick,inner sep=1pt,minimum size=3pt 
}, roundnode2/.style={circle, draw=black!60, fill=gray!50, very thick,inner sep=1pt,minimum size=3pt}
]

\draw[-,black,dotted, line width= 1.0pt](4.4,0)--(5.2,0);

\draw[-,black, line width= 1.5pt](5.73,0)--(7.3,0);
\draw[-,black, line width= 1.5pt](7.7,0)--(8.27,0);
\draw[-,black, line width= 1.5pt](8.73,0)--(10.27,0);
\draw[-,black, line width= 1.5pt](10.74,0)--(12.3,0);
\draw[-,black, line width= 1.5pt](12.7,0)--(13.3,0);
\draw[-,black, line width= 1.5pt](13.7,0)--(14.8,0);
\draw[-,black, line width= 1.5pt](15.2,0)--(15.8,0);
\draw[-,black, line width= 1.5pt](16.2,0)--(16.8,0);

\draw[-,black,dotted, line width= 1.0pt](17.3,0)--(18.0,0);

\node[roundnode]     at (5.5,0)    (-4){$\infty$};
\node[roundnode]     at (7.5,0)    (-3){$0$};

\node[roundnode]     at (8.5,0)    (-2){$\infty$};
\node[roundnode]     at (10.5,0)    (-1){$\infty$};
\node[roundnode2]     at (12.5,0)    (0){$0$};
\node[roundnode]     at (13.5,0)    (1){$1$};
\node[roundnode]     at (15.0,0)    (2){$0$};
\node[roundnode]     at (16.0,0)    (3){$0$};
\node[roundnode]     at (17.0,0)    (4){$2$};

\end{tikzpicture}
\end{center}
\caption{The central patch of the coloured point set $\varLambda_w$  derived from $w$. The location of the origin is signified by the shaded circle. The support of this coloured point set satisfies $\lambda\varLambda\subset \varLambda$ with $\lambda=\frac{5}{2}$.}
\label{fig:pointset}
\end{figure}
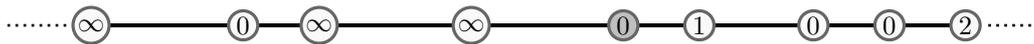

For primitive substitutions on finite alphabets, it is well known that the diffraction measure of a weighted Dirac comb supported on a Delone set $\varLambda$ arising from $\sub$ has non-trivial pure point component if and only if $\lambda$ is Pisot \cite{GK:diff}. Moreover, if $\widehat{\gamma}$ is pure point, $\varLambda$ must be a Meyer set \cite{LS:pp-Meyer}. 
Note however that a general Delone set need not be Meyer for it to have a non-trivial pure point component; see \cite{KS:Meyer,PFS:fusion-ILC} for the ``scrambled Fibonacci'' example, which arises from a fusion rule with finite local complexity.

Most proofs of existence of non-trivial pure point component (Bragg peaks in the diffraction setting, eigenvalues in the dynamical setting) for point sets with some form of hierarchical structure rely on the Diophantine properties of the return vectors, which generate the translation module. When the Delone set is of finite type, i.e., there are only finitely many distinct tile lengths, this $\Z$-module is finitely generated. For the example above, the translation module is infinitely generated and it is not clear whether the criteria for the existence of eigenvalues extend to this setting. This will be tackled in future work.

\section*{Acknowledgements}

\noindent
The authors would like to thank Michael Baake, Michael Coons, Natalie Priebe Frank, Franz G\"ahler, Christoph Richard and Nicolae Strungaru for fruitful discussions.

\bibliography{tilings}
\bibliographystyle{amsplain}

\end{document}